\documentclass[11pt]{article}

\usepackage[margin=1in]{geometry}
\usepackage{amsmath,amssymb,amsthm,mathtools}
\usepackage{booktabs,array,longtable}
\usepackage{enumitem}
\usepackage{xcolor}
\usepackage{hyperref}
\usepackage{cleveref}
\usepackage{tikz}
\usepackage[expansion=false]{microtype}
\usepackage{listings}
\usepackage[T1]{fontenc}
\IfFileExists{lmodern.sty}{\usepackage{lmodern}}{}
\usepackage{float}
\usepackage{textcomp}

\usetikzlibrary{arrows.meta,positioning,calc}

\hypersetup{
  colorlinks=true,
  linkcolor=blue!55!black,
  citecolor=green!45!black,
  urlcolor=blue!65!black,
  pdftitle={Braess' Paradox in Uniform Affine Grid Networks},
}

\newtheorem{theorem}{Theorem}[section]

\newtheorem{lemma}[theorem]{Lemma}
\newtheorem{corollary}[theorem]{Corollary}

\newtheorem{observation}[theorem]{Observation}
\newtheorem{definition}[theorem]{Definition}
\theoremstyle{remark}
\newtheorem{remark}[theorem]{Remark}
\theoremstyle{plain}

\newcommand{\Edir}{\mathcal{E}_{\rightarrow}}

\newcommand{\Lap}{\mathsf{L}}
\newcommand{\ee}{\mathbf{e}}
\newcommand{\E}{\mathcal{E}}

\numberwithin{figure}{section}

\title{\textbf{Braess' Paradox in Uniform Affine Grid Networks}}
\author{Andy Lu\thanks{Saratoga High School, Saratoga, California, USA.
Email: \texttt{anlu9183@gmail.com}} \quad and \quad Steven J. Miller\thanks{Department of Mathematics, Williams College,
Williamstown, Massachusetts, USA.
Email: \texttt{sjm1@williams.edu}}}
\date{\today}

\begin{document}
\maketitle

\begin{abstract}
\noindent Braess' Paradox is the phenomenon in which adding an edge to a congestion network increases total travel time. We study the paradox in directed rectangular grids where every edge shares the latency function $\ell(x)=ax+b$ and an added chord has latency $\ell_{*}(x) = cx+d$, where $a > 0$ and $b,c,d \ge 0$. Using an analogy with electrical networks, we bound the change in total travel time. We then give a necessary and sufficient condition for a chord to induce the paradox for some choice of nonnegative coefficients and compute the exact proportion of such chords in all grids with dimensions at most $100$. Such chords are scarce, and the fraction is maximized near an aspect ratio of $2:1$. We then improve the established $4/3$ upper bound on the Braess Ratio to one depending only on the grid dimensions, approaching $1.207$ on squares and $4/3$ on thin grids. Finally, we prove any ratio-maximizing chord must have zero latency.
\end{abstract}

\section{Introduction}

Road networks are generally designed under the intuition that providing drivers with more routes should reduce congestion and improve travel times. Braess' Paradox shows that this intuition can fail when each driver independently chooses the route that is fastest for them. In some networks, adding a new road changes these individual route choices in a way that increases the travel time experienced by every driver.

We model traffic as a continuous flow rather than as individual cars. 
This is a reasonable approximation when the number of drivers is large \cite{roughgarden2005}, since the route choice of one driver has a negligible effect on the overall congestion. Formally, when the players in an atomic splittable game are replicated, their equilibria converge to the continuous-flow equilibrium \cite{hauriemarcotte1985}.

We also consider traffic after the flow pattern has stabilized. This provides a reasonable approximation when demand is roughly constant over the period of interest.  Traffic enters the network at $s$ and leaves at $t$ at the same constant rate $q$, while at every other vertex the incoming and outgoing flow rates are equal. Thus, traffic does not continually accumulate at intersections. The steady-state assumption is an idealization, since real traffic varies over time and may not reach an equilibrium. However, Braess' Paradox is defined as a comparison between equilibrium costs, so we must work in the equilibrium setting to study it at all.

Consider a traffic network represented by a directed graph $G=(V,E)$ with a source vertex $s$, and a sink vertex $t$. Each edge contains some amount of traffic congestion, or flow; we set the total flow entering the graph at $s$ and leaving at $t$ as $q>0$. The congestion is a quantity proportional to the number of vehicles, rather than a literal vehicle count. 

 Assuming that each user traversing an edge under the same conditions experiences the same average speed, we may assign a latency function $\ell_e(x) \ge 0$ to each edge $e$, where $x$ denotes its flow. The term `latency' is conventional in selfish-routing literature \cite{roughgarden2005}.

Wardrop \cite{wardrop1952} introduced two fundamental traffic-flow configurations. The first is the \emph{User Equilibrium}, where no individual user can reduce their travel time by unilaterally changing routes. Thus for continuous latency functions and nonatomic flow, all used paths have the same travel time, while unused paths have an equal or greater travel time. The second configuration is the \emph{System Optimum}, in which traffic is routed to minimize the total travel time of all users. 

A flow pattern $f = (f_e)_{e \in E}$ specifies the nonnegative flow $f_e$ on each directed edge $e$. Its total cost---equivalently, its total travel time, as our latency functions measure time---is
\[
    C(f)\ :=\ \sum_{e\in E}f_e\ell_e(f_e).
\]
We write $f^G_{\mathrm{UE}}$ and $f^G_{\mathrm{SO}}$ for the User Equilibrium and System Optimum flows on graph $G$, respectively. 

Braess' Paradox concerns the change in User Equilibrium caused by adding
a new edge. Let $G+e$ denote the network obtained by adding an edge $e$
to $G$, while keeping the source, sink, and total demand unchanged.

\begin{definition}[Braess \cite{braess2005}]
    The Braess Ratio due to an added edge $e$ is defined as \[BR(e) \ :=\  \frac{C(f^{G + e}_{UE})}{C(f^G_{UE})}. \] If the added edge is clear, we write $BR$ for $BR(e)$. Braess' Paradox occurs when $BR(e) > 1$.
\end{definition}
For convenience, we write \[C_{\mathrm{new}} \ :=\  C\left(f_{UE}^{G + e}\right) \qquad \text{and} \qquad C_{\mathrm{old}} \ :=\  C\left(f_{UE}^{G}\right).\]

Previous work on Braess' Paradox has largely focused on broad theoretical aspects. Milchtaich \cite{Milchtaich2006} showed that a two-terminal network can admit the paradox if and only if it is not series-parallel, and Roughgarden proved that detecting the optimal subnetwork is NP-hard \cite{Roughgarden2006}. Steinberg and Zangwill \cite{SteinbergZangwill1983} provided early conditions for the paradox to occur, while Valiant and Roughgarden \cite{valiantroughgarden2010} and Chung, Young, and Zhao \cite{chungyoungzhao2012} proved it arises with high probability in random graphs and expanders. In contrast, we focus on a grid setting by characterizing chords capable of inducing the paradox. Grids are a natural setting for this question; street networks are frequently laid out on approximate rectangular grids, and the resulting structure is tractable enough to permit exact results.

Let $G_{m,n} = (V,E)$ be an $m \times n$ grid containing $mn$ cells. In particular, there are $(m+1) \times (n+1)$ vertices. The source $s$ and sink $t$ are taken as opposite corners
\[
s\ =\ (0,0),\qquad t\ =\ (m,n).
\]
We impose the simplifying restriction that grid roads must be directed toward the sink $(m,n)$. From a vertex $(i,j)$ the outgoing edges are
\[
(i,j)\to(i+1,j) \text{ when } i <m  \qquad \text{and} \qquad
(i,j)\to(i,j+1) \text{ when } j < n,
\]
and every edge flow satisfies $f_e\ge 0$: no driver may traverse a road backwards. 

Henceforth assume each grid edge $e \in E$ has the same affine latency
\[
\ell_e(x)\ =\ ax+b, \qquad x\ \ge\ 0, \]
for constants $a>0,\quad b\geq 0.$ Meanwhile, let the added edge, or \emph{chord}, have latency \[\ell_*(x)\ =\ cx+d, \qquad x\ \ge\ 0,\] with $c,d \ge 0$. In particular, $b, d\ \ge0$ guarantee nonnegative latency at zero flow, while $a > 0$ and $c \ge0$ guarantee travel time does not decrease when traffic increases. We require $a > 0$ to exclude the degenerate case in which Braess' Paradox cannot occur. 

The use of identical latencies models a grid in which roads have comparable geometry. Meanwhile, the chord's latency may differ to accommodate a road with different characteristics, such as a different length. 

We restrict our added chord to be directed $u \rightarrow v, u\neq v$ such that $uv$ is not an edge of the underlying undirected grid. Figure~\ref{fig:grid-example} illustrates the construction for $G_{3,3}$.

\begin{figure}[H]
    \centering
    \includegraphics[scale=0.88]{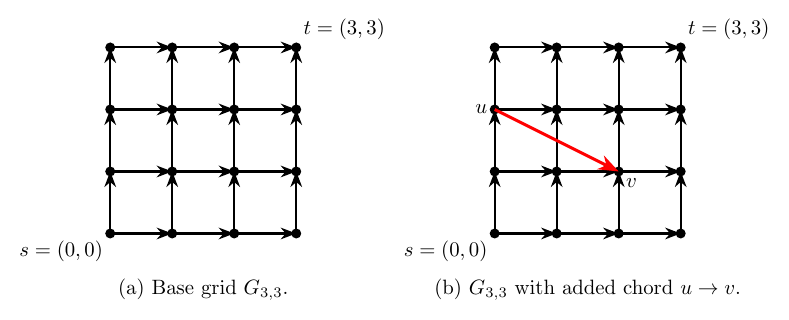}
    \caption{An example of the directed rectangular-grid configuration.}
    \label{fig:grid-example}
\end{figure}

\begin{definition}
Let the \emph{level} of a vertex $(i,j)$ be \[\lambda(i,j) \ :=\ i + j. \]
For an added edge $u \rightarrow v$, the level gain is then \[k_{u,v} \ :=\  \lambda(v) - \lambda(u).\] When $u \rightarrow v$ is fixed, we write $k_{u,v}$ as $k$.
\end{definition}

A feasible flow pattern routes total demand $q$ from $s$ to $t$: all edge flows are nonnegative, flow is conserved at every vertex other than $s$ and $t$, and the net outflow from $s$ and net inflow to $t$ both equal $q$. For a feasible flow pattern $f$, define the Beckmann Potential \cite{beckmann1956}
\[
Z(f)
\ :=\ 
\sum_{e\in E}\int_0^{f_e}\ell_e(y) dy.
\]
The integral is chosen so that an infinitesimal change in edge flow causes a change in $Z$ of exactly that edge's latency. Thus the conditions minimizing $Z$ correspond to the conditions of the User Equilibrium.

\begin{theorem}[Beckmann, McGuire, Winsten, 1956 \cite{beckmann1956}]

Suppose every latency function is continuous and nondecreasing. Then the minimum Beckmann Potential corresponds exactly to the User Equilibrium flow pattern.  If every latency function is strictly increasing, the equilibrium edge flow is unique.
\end{theorem}

For affine latency functions $\ell_e(f_e)=a_ef_e+b_e$, the potential becomes
\[
Z(f)
\ =\ 
\sum_{e\in E}
\left(
\frac{a_e}{2}f_e^2+b_ef_e
\right).
\]

Working in the directed grids defined above, we establish the following. 

We first provide an analogy between traffic and electrical resistor networks. Utilizing results on the latter, we introduce the orientation-constrained energy. This yields a bound on the increase in total travel time valid at every demand and every choice of latency coefficients, and shows that a negative electrical voltage drop is necessary for the paradox. 

Then, we prove that an added chord may induce Braess' Paradox in the grids defined above, for some choice of coefficients $b,c,d \ge 0$, if and only if its electric voltage drop is negative and its level gain is positive. This allows us to compute the proportion of Braess-capable chords in grids with dimensions at most 100; numerical results are provided in the \hyperref[app:results]{appendix}. Although previous prevalence results \cite{valiantroughgarden2010, chungyoungzhao2012} are probabilistic and asymptotic, the grid structure permits an exact count: square grids smaller than $7 \times 7$ admit no such chord, and the proportion peaks near aspect ratio $2:1$ at roughly $0.0414$. 

Results from Roughgarden and Tardos \cite{roughgardentardos2002} imply the Braess Ratio is at most $4/3$ for affine latencies on arbitrary networks; we sharpen this on grids, proving an upper bound of \[BR\ \le \ 1+\frac{1}{\frac{m+n}{\max(m,n)-1} + 2\sqrt{\frac{m+n}{\max(m,n)-1}}}\] which approaches $1.207$ on square grids and $4/3$ only as the grid becomes thin. Writing $r = n/m$ and assuming $m \ge n$, the bound admits the expansion \[\frac43 - \frac29 r + \frac{19}{108}r^2 + O(r^3) \qquad \text{as }r \rightarrow 0.\] We also prove that a chord maximizes the Braess Ratio only if it has zero latency.

The organization of the paper is as follows. Section~\ref{sec:criteria} develops the analogy with electrical networks, introduces the orientation-constrained energy, and proves the Braess criterion. We then analyze the consequences of the criterion in Section~\ref{sec:consequences}. Section~\ref{sec:bounds}
proves the grid-dependent ratio bound and optimal chord latency. Lastly, Section~\ref{sec:future} suggests possible extensions.

\section{Criteria for Braess' Paradox}

\begin{observation}
\label{thm:linear-term}
    In the directed grid, every route from $s$ to $t$ uses exactly $m+n$ edges. Consequently, for any feasible flow routing total demand $q$ through the base grid $G_{m,n}$,
\[\sum_{e\in E}f_e\ =\ q(m+n).\]
\end{observation}

As a result, the Beckmann Potential in the base grid is
\[
Z(f_{G_{m,n}})
\ =\ 
\frac a2\sum_{e\in E}f_e^2+bq(m+n),
\]
whose second term is fixed regardless of flow pattern. Hence the base-grid equilibrium is obtained by minimizing
\[
a\sum_{e\in E}f_e^2,
\] which is twice the quadratic term of the Beckmann Potential.

\label{sec:criteria}
\subsection{Basic Results}

\begin{lemma}[Augmented Linear Term]
\label{lem:linear_term}
 If $ u\rightarrow v$ carries flow $z$, then the flow on the original grid satisfies
\begin{equation}
\sum_{e\in E}f_e\ =\ q(m+n)-k_{u,v}z.
\end{equation}
\end{lemma}
\begin{proof}
    Every unit of flow along a grid edge increases its level by $1$. Meanwhile, every unit of flow along the additional edge increases the level by $k$. The total accumulated level gain must be $q(m+n)$, hence \[q(m+n) \ = \    kz + \sum_{e \in E} f_e,\] which rearranges to the desired identity. 
\end{proof}

\begin{lemma}[Direct edge]
\label{thm:direct-edge}
Assume that each edge in a network $G$ has a continuous nondecreasing latency function. Then adding a direct road from $s$ to $t$ cannot increase the total travel time.
\end{lemma}

\begin{proof}
We have established that every route taken in the User Equilibrium takes the same amount of time. Let this time be $\tau_H(q)$ for total traffic flow $q$ on graph $H$. Hall \cite{hall1978} proved in 1978 that $\tau_H(q)$ must be nondecreasing in $q$.

Let the augmented graph be $G'$, with flow $z$ passing through the additional edge (and $q-z$ through the original graph). 

If $z = q$, the original graph is unused. Then for any $\varepsilon \in (0,q]$, consider a path from $s$ to $t$ through $G$ used at equilibrium under demand $\varepsilon$. At infinitesimal flow, this path has latency at most $\tau_G(\varepsilon)$. Then the travel time through the direct edge is at most $\tau_G(\varepsilon)$; otherwise a user may switch to a path on the grid while decreasing their travel time. By monotonicity of $\tau_G$, \[\tau_{G'}(q)\  \le\  \tau_G(\varepsilon) \ \le\  \tau_G(q)\] as desired. 

Otherwise, if $z < q$ we have \[\tau_{G'}(q) \ =\  \tau_G(q-z) \ \le \  \tau_G(q).\] As the time taken by each user in the new graph cannot exceed that of the original, the total travel time cannot have increased either.
\end{proof}

\begin{lemma}[Linear Latencies]
\label{thm:linear-latency}
Consider a network in which every original and added edge has linear latency
\[
\ell_e(x)\ =\ a_ex,
\qquad a_e\geq 0.
\]
Adding another linear latency edge cannot increase the total travel time at User Equilibrium.
\end{lemma}

\begin{proof}
For linear costs,
\[
Z(f)\ =\ \frac12\sum_e a_ef_e^2
\ =\ \frac12C(f).
\]
In other words, the total travel cost is twice the Beckmann Potential. As $f$ remains a feasible flow pattern in the augmented graph, the new potential cannot be greater than the original. Therefore Braess' Paradox cannot occur in networks with linear edge latencies.
\end{proof}

\subsection{Electricity}
The quadratic portion of the Beckmann Potential can be naturally interpreted using electrical networks \cite{roughgarden2005}.

A resistor is a well-known electrical circuit element that opposes current flow. For an edge with resistance $R$ and current $I$, the voltage drop in the direction of the current is given by Ohm's Law \cite{griffiths2013} \[V \ = \ IR,\] and the power is \[P \ = \ I^2R.\] Throughout the paper, we use the term ``energy'' for the quadratic power dissipation term.  

Due to Ohm's Law, a two-terminal network of resistors can be represented by an effective resistance that produces the same voltage drop and power dissipation.

\begin{theorem}[Thomson's Principle]
    \emph{Thomson's Principle} \cite{doylesnell1984} states that, among all current distributions with the amounts of current entering and leaving at each vertex held fixed, the electrical current minimizes the total dissipated power.
\end{theorem}

\begin{remark}
\label{obs:resistor}
    Lemma~\ref{thm:linear-latency} shows that Braess' Paradox is impossible in a pure resistor network.
\end{remark}

Following Spielman \cite{spielman2019}, we introduce the incidence matrix and the weighted graph Laplacian.

The directed incidence matrix $B\in\mathbb{R}^{|V|\times |E|}$ satisfies \[B_{v, v_1 \rightarrow v_2} \ :=\  \begin{cases}
  1,  & \text{if } v=v_1,\\
 -1,  & \text{if } v=v_2,\\
  0,  & \text{otherwise}.
\end{cases}\] 
For an undirected graph $H=(V,E_H)$ with resistance $r_e$ on edge $e$, the weighted graph
Laplacian $\Lap_H$ is defined by
\[
\left(\Lap_H\right)_{ij}
\ := \ 
\begin{cases}
\displaystyle
\sum_{ik\in E_H} \frac{1}{r_{ik}},
    & \text{if }i=j, \\
-\frac{1}{r_{ij}},
    & \text{if }ij\in E_H, i\neq j \\
0,
    & \text{otherwise}.
\end{cases}
\]
Although our traffic grid is directed, whenever we form a graph
Laplacian, we use the underlying undirected graph obtained by ignoring
the directions of its edges. Let the grid Laplacian be $\Lap$. On our traffic grid with edge latency $ax+b$, we treat $a$ as the resistance on an edge; then we may also represent $\Lap$ as $\frac1a BB^{\mathsf{T}}$.

The current injection vector is the vector $p\in \mathbb{R}^{|V|}$ such that $p_i$ denotes the current flowing into the graph at $i$; a negative $p_i$ represents current flowing out. (We write $p$ rather than $b$, which is reserved for the latency intercept.) As the total injected current must equal the total removed current, we require \[\sum_{i\in V} p_i\ =\ 0.\] Following Ghosh, Boyd, and Saberi \cite{ghosh2008}, the voltage vector $\phi$ satisfies Kirchhoff's law \[\Lap\phi\ =\ p. \] The solution $\phi$ is only unique up to addition of a constant; the canonical solution is given by \[ \phi\ =\ \Lap^+p, \] where $\Lap^+$ denotes the \emph{Moore--Penrose pseudoinverse} of $\Lap$.

The electrical power dissipated by the injection vector $p$ is then
\begin{align*}
    \mathcal{E}(p) \ &:= \sum_{ij \text{ edge}} \frac{1}{r_{ij}}(\phi_i-\phi_j)^2 \\
    &= \phi^{\mathsf{T}}\Lap\phi \\
    &= p^{\mathsf{T}}\Lap^{+}p.
\end{align*}

Write $\ee_v \in \mathbb{R}^{|V|}$ as the standard basis vector with 1 in coordinate $v$ and 0 everywhere else. 

\begin{lemma}[Potential Monotonicity]
\label{lem:monotone}
Let $\phi=\Lap^{+}q(\ee_s-\ee_t)$ be the potential of the electrical flow on the $m \times n$ base grid with $s = (0,0)$ and $t = (m,n)$. Then $\phi$ is \emph{strictly} decreasing along every directed grid edge:
\[
\phi_{i,j}\ >\ \phi_{i+1,j}
\qquad\text{and}\qquad
\phi_{i,j}\ >\ \phi_{i,j+1}.
\]
\end{lemma}

\begin{proof}
Consider the horizontal voltage drops
\[
\psi_{i,j}\ :=\ \phi_{i,j}-\phi_{i+1,j},
\qquad
0\leq i<m,\quad 0\leq j\leq n.
\] 
For convenience, set
\[
\psi_{-1,j}\ =\ \psi_{m,j} \ = \ \psi_{i, -1}\ =\ \psi_{i,n+1}\ =\ 0.
\]

Subtracting Kirchhoff's equation at $(i+1,j)$ from that at $(i,j)$ gives
\begin{equation}
\label{eq:mon}
d_j\psi_{i,j}
\ = \
\psi_{i-1,j}+\psi_{i+1,j}
+\psi_{i,j-1}+\psi_{i,j+1}
+aq\eta_{i,j},
\end{equation}
where 
\[
d_j\ :=\ 
\begin{cases}
4,&0<j<n,\\
3,&j=0\text{ or }j=n,
\end{cases}
\]
and $\eta_{i,j}=1$ when $(i,j)=(0,0)$ or $(m-1,n)$ and is $0$ otherwise. Thus each horizontal drop is an average of its neighboring drops, except for the positive term at the two corner edges.

Let $M$ be the minimum voltage drop
\[
M\ :=\ \min_{i,j}\psi_{i,j}.
\]

Suppose for the sake of contradiction that $M < 0$, and consider an edge attaining this
minimum. If this edge lies on a boundary, $i=0$ or $i = m-1$, each neighboring drop is at least $M$, one of which is $0$. Therefore \eqref{eq:mon} is impossible. 

However, if the edge doesn't lie on a boundary, equality in \eqref{eq:mon} requires, in particular, the two horizontally adjacent voltage drops to be $M$ as well. Repeating this argument until we reach a boundary again yields a contradiction. 

If $M = 0$, every term on the right-hand side of \eqref{eq:mon} is nonnegative, so equality forces each term to be 0. Repeating this argument until we reach a corner, where $\eta_{i,j} = 1$, gives a contradiction. Hence $M > 0$.

The same argument applies to vertical drops, and therefore
\[
\phi_{i,j}\ >\ \phi_{i,j+1}
\] as well. 
By Ohm's law, every base-grid edge consequently carries strictly positive current in its directed orientation.
\end{proof}

\subsection{Beckmann Potential of the Grid and the Directed Energy}
\label{sec:beckmann-grid}

Suppose that the additional directed edge $u \rightarrow v$ carries a nonnegative flow $z$ and has latency $\ell_{*}(x) = cx+d$. 

Every feasible flow pattern is a combination of paths from $s$ to $t$ totaling $q$ units, and directed flow cycles. Then, if $z > q$ such a directed cycle must exist through the added edge; reducing the flow uniformly around the cycle keeps the configuration feasible while strictly decreasing the potential. Therefore no equilibrium has $z > q$. Henceforth assume $0 \le z \le q$.

The chord diverts a flow of $z$ from $s$ to $t$ via $u\rightarrow v$. Hence the net injection vector that the grid must satisfy is
\[
p(z) := q(\ee_s-\ee_t)-z(\ee_u-\ee_v).
\]

Let $B$ be the directed incidence matrix of the grid. The quadratic energy at chord flow $z$ is the minimum dissipated power subject to routing this injection \emph{with nonnegative road flows}:
\begin{equation}
\label{eq:Edir}
\Edir(z)
\ :=\ 
\min\left\{ a\|f\|_2^2 : \ Bf = p(z),\ f\ge 0 \right\}
\end{equation}
We call $\Edir$ the \emph{orientation-constrained energy}. Unlike $\E(p(z))$, which minimizes over electrical currents on the undirected grid, $\Edir(z)$
restricts every edge flow to its prescribed nonnegative direction.

Note that the feasible set is nonempty for every $z \in [0,q]$. Indeed, we may take $q-z$ units of flow from $s$ to $t$, $z$ units from $s$ to $u$, and $z$ units from $v$ to $t$. The resulting net injection vector is \[p(z) \ =\  (q-z)(\ee_s - \ee_t) + z(\ee_s - \ee_u) + z(\ee_v - \ee_t).\]

The minimum is attained as the feasible flow set is closed, while $a\|f\|_2^2$ is continuous and tends to infinity as $\|f\|_2$ increases. Strict
convexity implies that the minimizing flow is unique. For each chord flow $z\in[0,q]$, let $f(z)$ be this minimum-energy grid flow pattern and $f_e(z)$ its flow on edge $e$, so that
\[
    \Edir(z)\ =\ \sum_{e \in E}af_e(z)^2 \ =\  a\|f(z)\|_2^2.
\]

We also define the electrical quantities
\begin{align*}
\label{eq:definitions}
V_{uv}\ &:=\ q(\ee_s-\ee_t)^{\mathsf T}\Lap^{+}(\ee_u-\ee_v) \\
R_{uv}&:=(\ee_u - \ee_v)^{\mathsf T}\Lap^+(\ee_u - \ee_v) \\
R_{st} &:= (\ee_s - \ee_t)^{\mathsf T}\Lap^+(\ee_s - \ee_t),\end{align*} where $V_{uv}$ is the voltage drop from $u$ to $v$ when current $q$ flows from $s$ to $t$, while $R_{uv}$ and $R_{st}$ denote the effective resistances between $u,v$ and $s,t$ respectively. All three quantities are taken on the underlying undirected $m \times n$ grid with edge resistance $a$.  

\begin{definition}[Flow regions]
For all $z\in[0,q]$, let
\[
E_+(z)\ :=\ \{e\in E:f_e(z)>0\}
\]
denote the set of original grid roads carrying strictly positive flow in the unique potential-minimizing flow pattern when the chord carries flow $z$. 

A \emph{flow region} is a maximal open interval contained in $(0,q)$ for which $E_+(z)$ is constant. For all $z$ in such a region, every edge in $E_+(z)$ carries positive flow while the
roads outside $E_+(z)$ remain at zero flow.

The \emph{first flow region} is then the interval $(0, \overline{z})$, where $\overline{z}$ is the first value at which a grid edge reaches 0 flow. We take $\overline{z} = q$ if no edge reaches 0 flow over $z \in [0,q]$. 
In particular, 
\[E_+(z)\ =\ E\qquad\text{for }0\le z<\overline{z}.\]
\end{definition}

The unconstrained electrical flow is affine, and hence continuous over $z$. Every original grid road carries strictly positive flow at $z=0$ by Lemma~\ref{lem:monotone}. Therefore the first region is nondegenerate. 

On the first region, the orientation-constrained problem reduces to the ordinary electrical minimum-energy problem with injection vector $p(z)$. The quadratic energy term for the grid then satisfies
\begin{equation}
\label{eq:quadenergy}
\Edir(z)
\ =\  \E(p(z)) \ =\  
\left[q(\ee_s-\ee_t)-z(\ee_u-\ee_v)\right]^\mathsf{T}
\Lap^+
\left[q(\ee_s-\ee_t)-z(\ee_u-\ee_v)\right].
\end{equation}
This becomes 

\begin{align*}
\Edir(z)
\ =
&\ q^2(\ee_s-\ee_t)^\mathsf{T}\Lap^+(\ee_s-\ee_t) \\
&-qz (\ee_s-\ee_t)^\mathsf{T}\Lap^+(\ee_u-\ee_v) \\
&-qz (\ee_u-\ee_v)^\mathsf{T}\Lap^+(\ee_s-\ee_t) \\
&+ z^2(\ee_u-\ee_v)^\mathsf{T}\Lap^+(\ee_u-\ee_v).
\end{align*}
We now substitute the expressions for $R_{st}, R_{uv}, V_{uv}$
to obtain
\[
\Edir(z)
\ =\ 
q^2R_{st}
-
2V_{uv} z
+
R_{uv}z^2.
\]

Beyond $\overline z$, the expression is only a lower bound. Since the directed feasible flows form a subset of all signed electrical flows,
\begin{equation}
\label{eq:relaxation-bound}
\Edir(z) \ \ge \  q^2R_{st}-2V_{uv}z+R_{uv}z^2
\qquad\text{for all }z.
\end{equation} We may also express $\Edir(z)$ exactly in terms of the graph Laplacian.

\begin{theorem}[Directed Energy]
\label{thm:structure}
The energy term $\Edir(z)$ satisfies the following.
\begin{enumerate}[label=(\roman*)] 
    \item $\Edir(z)$ is strictly convex over $[0,q]$ and quadratic on each flow region.
    \item On every interval where the subgraph $A = (V, E_+(z))$ of positive-flow roads is constant,
    \[
    \Edir(z) \ =\  p(z)^{\mathsf T} \Lap_A^{+}  p(z),
    \qquad p(z)\ =\ q(\ee_s-\ee_t)-z(\ee_u-\ee_v),
    \]
    where $\Lap_A$ denotes the Laplacian of $A$.
\end{enumerate}
\end{theorem}

\begin{proof}

Take distinct $z_1,z_2\in[0,q]$ with optimizers $f(z_1),f(z_2)$, and let $r\in(0,1)$.
Because $B$ is linear,
\[
B\left(rf(z_1)+(1-r)f(z_2)\right)\ =\ r p(z_1)+(1-r) p(z_2)\ =\ p\left(rz_1+(1-r)z_2\right),
\]
and $rf(z_1)+(1-r)f(z_2)\ge0$, so this combination is feasible at $rz_1+(1-r)z_2$.
Hence
\begin{align*}
\Edir\left(rz_1+(1-r)z_2\right)
\ &\le \ a\|rf(z_1)+(1-r)f(z_2)\|_2^2 \\
&< \ ra\|f(z_1)\|_2^2+(1-r) a\|f(z_2)\|_2^2 \\
&= \ r\Edir(z_1)+(1-r)\Edir(z_2).
\end{align*}
The strict inequality holds because $z_1 \neq z_2$ and so $f(z_1)\ne f(z_2)$. Thus $\Edir$ is strictly convex.

Fix a flow region and let $A=(V,E_+(z))$ be its (constant) active subgraph.
Throughout the region every active edge carries strictly positive flow. The minimization therefore reduces to the unconstrained electrical problem on $A$, whose solution is the potential $\phi=\Lap_A^{+}p(z)$ with energy
\[
\Edir(z)\ =\ \phi^{\mathsf T}\Lap_A\phi\ =\ p(z)^{\mathsf T}\Lap_A^{+}p(z).
\]
Since $p(z)$ is affine in $z$, this is a quadratic in $z$. 
\end{proof}

\begin{remark}
    Although $A$ may be disconnected, this causes no problems. The total injection on each connected component of $A$ is zero. Thus, Kirchhoff's  equations are solvable on each component. Combining these solutions gives the canonical solution on $A$ shown above. 
\end{remark}

By Lemma~\ref{lem:linear_term}, the intercept contribution of the grid is $b(q(m+n)-kz)$. Hence the Beckmann Potential, expressed as a function of $z$, is
\begin{equation}
\label{eq:psi}
Z(z)
\ =\  \frac12 \Edir(z) + b\left(q(m+n) - kz\right) + \frac{c}{2}z^2+dz,
\qquad 0\le z\le q.
\end{equation} The equilibrium is found by minimizing $Z$ over all $0\le z \le q$. Let this minimum be achieved at $z = z^*$.

On the first region
\begin{equation}
\label{eq:first_potential}
Z(z) \ =\  \left(\frac{1}{2}R_{uv}+\frac{c}{2}\right)z^2- (V_{uv} + bk - d)z + \frac{1}{2}q^2R_{st} + bq(m+n).
\end{equation}

\subsection{Braess Criterion}
\label{sec:braess-crit}

We first identify when the new edge carries positive equilibrium flow.

\begin{corollary}[Edge usage]
\label{cor:use}
The additional edge has positive flow if and only if
\[
V_{uv}+bk\ >\ d.
\]
\end{corollary}

\begin{proof}
By Theorem~\ref{thm:structure}(i), $\Edir$ is strictly convex, hence so is $Z$; it therefore has a unique minimizer $z^*$ on $[0,q]$. Therefore $z^*$ is positive if and only if the right-hand derivative at 0 is negative. This derivative is \[-V_{uv}-bk+d,\] by \eqref{eq:first_potential},  which is negative precisely when $V_{uv}+bk>d$. Again note that $V_{uv}$ is the \emph{original grid} voltage.
\end{proof}

\begin{corollary}[First-region equilibrium]
\label{cor:flow}
Let $z^*$ be the minimizer of $Z(z)$, and
\[
z_{\mathrm{cand}}\ :=\ \frac{V_{uv}+bk-d}{R_{uv}+c}.
\]
Then:
\begin{enumerate}[label=\textup{(\roman*)}]
\item if $z_{\mathrm{cand}} \le 0$ then the equilibrium is that of the base grid;
\item if $0<z_{\mathrm{cand}}\le \overline z$ then the equilibrium satisfies $z^*=z_{\mathrm{cand}}$;
\item if $\overline{z} = q < z_{\mathrm{cand}}$ then $z^* = q$;
\item if $\overline{z} < z_{\mathrm{cand}}$ and $\overline{z} < q$, the first-region formula does not determine $z^*$: the equilibrium does not lie strictly inside the first region.
\end{enumerate}
\end{corollary}

\begin{proof}

If $z_{\mathrm{cand}} \le 0$, we have $V_{uv} + bk-d \le 0$. By Corollary~\ref{cor:use} the additional edge must be unused. Therefore the Beckmann-minimizing flow is exactly the base-grid equilibrium.

Recall the first-region Beckmann Potential \[Z(z) \ =\  \left(\frac{1}{2}R_{uv}+\frac{c}{2}\right)z^2- (V_{uv} + bk - d)z + \frac{1}{2}q^2R_{st} + bq(m+n).\] The expression $\frac{V_{uv}+bk - d}{R_{uv}+c}$ is therefore the value of $z$ that minimizes this quadratic.

In case (ii) we first consider $z_{\mathrm{cand}} = \overline{z} = q$, so that $Z'_-(q)=0$. By convexity, $Z$ is nonincreasing up to $q$, so $q$ is the minimizer. Next, if $z_{\mathrm{cand}} = \overline{z} < q$, by continuity we may still use the first-region potential at $\overline{z}$, where it has left derivative equal to 0. Since $Z$ is continuous and convex,  
\[ Z'_+(\bar z)\ \ge\  Z'_-(\bar z)\ =\ 0. \]
Therefore $\bar z$ minimizes $Z$ on $[0,q]$.

If $0 < z_{\mathrm{cand}} < \overline{z},$ it is contained in the first region, so it is the global minimizer of $Z$.

In case (iii), the left-hand derivative of $Z$ is negative up to $\overline z$, so the minimum potential lies at exactly $z^* = q$. 
However, when $\overline{z} \neq q$ and $\overline{z} < z_{\mathrm{cand}}$ as in case (iv), the minimum may lie beyond $\overline{z}$. 
\end{proof}

We first record the closed-form change in total travel time when the equilibrium is interior to the first region, then establish a global worst-case inequality that governs all regions.

\begin{theorem}[First-region time change]
\label{thm:local-braess}
Consider an $m \times n$ grid with source $s = (0,0)$ and sink $t = (m,n)$ and edges of identical latency $\ell(x) = ax+b$. Add a directed edge $u \rightarrow v$ with latency $\ell_{*}(x) = cx+d$. Suppose the equilibrium chord flow satisfies
\[
0\ <\ z^*\ <\ \overline{z}.
\]
Then
\[
z^*\ =\ \frac{V_{uv}+bk-d}{R_{uv}+c}\]
and the difference in total travel times is \[
C_{\mathrm{new}}-C_{\mathrm{old}}\ =\ -V_{uv}z^*.\]
Such a chord induces Braess' Paradox if and only if $V_{uv}<0$.
\end{theorem}

\begin{proof}
The exact value of $z^*$ was proved in Corollary~\ref{cor:flow}.

The total cost on the new network is
\begin{align*}
    C_{\mathrm{new}} \ &=\  \sum_{e \in E} f_e(z^*)(af_e(z^*) + b)+z^*(cz^*+d) \\
    &= \ \Edir(z^*) + b\left(q(m+n) - kz^*\right) + c(z^*)^2+dz^*.
\end{align*}
Subtracting $C_{\mathrm{old}}=q^2R_{st}+bq(m+n)$ and using the first-region form $\Edir(z^*)=q^2R_{st}-2V_{uv} z^*+R_{uv}(z^*)^2$ gives
\[
C_{\mathrm{new}}-C_{\mathrm{old}}
\ =\ 
-2V_{uv} z^*+R_{uv}(z^*)^2
-bkz^*+c(z^*)^2+dz^* .
\]
Substituting the expression for $z^*$ yields $C_{\mathrm{new}}-C_{\mathrm{old}}=-V_{uv}\cdot z^*$. Braess' Paradox occurs when this difference is positive, corresponding to $V_{uv} < 0$. 
\end{proof}

The conditions $V_{uv}<0$ and $V_{uv}+bk>d$ are sufficient for the paradox only when the User Equilibrium remains interior to the first region. To obtain conclusions valid in all regions we use the following.

\begin{theorem}[Global travel-time change]
\label{thm:global-cost-bound}
Suppose the added edge carries equilibrium flow $z^*\in(0,q].$
Then the difference in total travel times is bounded by
\[
C_{\mathrm{new}}-C_{\mathrm{old}}
\ \le\ 
-V_{uv}z^*.
\]
Consequently, $V_{uv}<0$ is necessary for Braess' Paradox.
\end{theorem}

\begin{proof}

The conservation constraints give
\[
    B(f(0)-f(z))\ =\ p(0)-p(z)\ =\ z(\ee_u-\ee_v).
\]
Let
\[
    \phi_0\ =\ q\Lap^+(\ee_s-\ee_t)
\]
be the base-grid potential. Clearly $f(0)$ is simply the unconstrained electrical flow on the base grid. Since
\[
    V_{uv}\ =\ \phi_0^\mathsf{T}(\ee_u-\ee_v)
    \qquad\text{and}\qquad
    B^\mathsf{T}\phi_0\ =\ af(0),
\]
we have
\begin{equation}
    zV_{uv}
    \ =\ \phi_0^\mathsf{T}B(f(0)-f(z))
    \ =\ a f(0)^\mathsf{T}(f(0)-f(z)).
    \label{eq:voltage-flow}
\end{equation}

We next bound the left derivative of $\Edir$. For $0<\varepsilon<1$, the flow
\[
    (1-\varepsilon)f(z)+\varepsilon f(0)
\]
is nonnegative and satisfies
\[
\begin{aligned}
    B\left((1-\varepsilon)f(z)+\varepsilon f(0)\right)
    \ &=\ (1-\varepsilon)p(z)+\varepsilon p(0) \\
    &= \ p\left((1-\varepsilon)z\right).
\end{aligned}
\]
It is therefore feasible at chord flow $(1-\varepsilon)z$. By the minimality of
$f\left((1-\varepsilon)z\right)$,
\[
\begin{aligned}
    \Edir\left((1-\varepsilon)z\right)
    \ &\leq\ 
    a\left\|(1-\varepsilon)f(z)+\varepsilon f(0)\right\|_2^2 \\
    \ &=\    
    \Edir(z)
    +2a\varepsilon f(z)^\mathsf{T}(f(0)-f(z))
    +a\varepsilon^2\|f(0)-f(z)\|_2^2.
\end{aligned}
\]
Rearranging, dividing by $\varepsilon z$, and letting $\varepsilon\rightarrow0^+$ gives
\[
    {\Edir'}_-(z)
    \ \geq\ 
    \frac{2a}{z}f(z)^\mathsf{T}(f(z)-f(0)).
\]
Consequently,
\begin{align}
    \Edir(z)-\Edir(0)-\frac{z}{2}{\Edir'}_-(z)
    \ &\leq\ 
    a\|f(z)\|_2^2-a\|f(0)\|_2^2
    -a f(z)^\mathsf{T}(f(z)-f(0)) \notag\\
    \ &=\ 
    a f(0)^\mathsf{T}(f(z)-f(0)) \notag\\
    \ &=\ 
    -V_{uv}z,
    \label{eq:energy-bound}
\end{align}
where the final equality follows from \eqref{eq:voltage-flow}.

Since $z^*>0$ minimizes $Z$ on $[0,q]$, its left derivative satisfies
\[
    Z'_-(z^*)
    \ =\ 
    \frac{1}{2}{\Edir'}_-(z^*)-bk+cz^*+d
    \ \leq\  0.
\]
This condition applies whether $z^*<q$ or $z^*=q$. Therefore,
\[
    -bk+cz^*+d
    \ \leq\ 
    -\frac{1}{2}{\Edir'}_-(z^*).
\]
Using the expression for the change in total travel time,
\[
\begin{aligned}
    C_{\mathrm{new}}-C_{\mathrm{old}}
    \ &=\ 
    \Edir(z^*)-\Edir(0)
    +z^*\left(-bk+cz^*+d\right)\\
    \ &\leq\ 
    \Edir(z^*)-\Edir(0)
    -\frac{z^*}{2}{\Edir'}_-(z^*)\\
    \ &\leq\ 
    -V_{uv}z^*,
\end{aligned}
\]
where the final inequality follows from \eqref{eq:energy-bound}.

Since Braess' Paradox requires
\[
    C_{\mathrm{new}}-C_{\mathrm{old}}\ >\ 0
\]
and $z^*>0$, it follows that $V_{uv}<0$.
\end{proof}

The above shows why a complete criterion for fixed latency parameters is difficult. Once the chord flow passes beyond the first flow region, some original roads become inactive, the electrical formula for $\Edir(z)$ is no longer exact, and the sign condition $V_{uv} < 0$ alone does not determine whether the total travel time increases. However, we may find a necessary and sufficient condition for an added edge to induce Braess' Paradox over \emph{some} choice of $b,c,d \ge 0$. 

\begin{theorem}[Braess capability]
\label{thm:capability}
Fix a directed grid $G_{m,n}$, slope $a>0$ and demand $q > 0$. Call an additional edge $u \rightarrow v$ \emph{Braess-capable} if there exist $b,c,d \ge 0$ such that, if the grid edges have latency function $ax+b$ and the added edge has latency function $cx+d$, the resulting Braess Ratio exceeds 1. 

Then $u \rightarrow v$ is Braess-capable if and only if
\[
    V_{uv}\ <\ 0 \quad\text{and}\quad k \ > \ 0.
\]
\end{theorem}

\begin{proof}
First suppose that the chord induces Braess' Paradox for some
$b,c,d\geq 0$. We have already shown $V_{uv} < 0$. It remains to show that $k > 0$. 

By the edge-usage criterion, positive
chord flow requires
\[
    V_{uv}+bk\ >\ d.
\]
Suppose instead that $k\leq 0$. Since $b\geq 0$,
\[
    bk\ \leq\  0.
\]
Together with $V_{uv}<0$, this gives
\[
    V_{uv}+bk<0\ \leq\  d,
\]
which is impossible. This proves the necessity.

Conversely, suppose that
\[
    V_{uv}\ <\ 0
    \qquad\text{and}\qquad
    k\ >\ 0.
\]

Choose any $\varepsilon$ satisfying
\[
    0\ <\ \varepsilon\ <\ R_{uv}\overline{z}.
\] 
Here $R_{uv}$ must be positive as $u \neq v$. 

Set
\[
    c\ =\ 0,
    \qquad
    d\ =\ 0,
    \qquad
    b\ =\ \frac{-V_{uv}+\varepsilon}{k}
\]
so that
\begin{align*}
    V_{uv}+bk-d
    \ &= \ 
    V_{uv}
    +
    k\left(\frac{-V_{uv}+\varepsilon}{k}\right) \\
    \ &=\ 
    \varepsilon
    >0.
\end{align*}
Hence the chord attracts positive flow. Its first-region equilibrium
candidate is
\[
    z_{\mathrm{cand}}
    \ = \
    \frac{V_{uv}+bk-d}{R_{uv}+c}
    \ = \ 
    \frac{\varepsilon}{R_{uv}}.
\]
As this expression is strictly less than $\overline{z}$, the equilibrium lies inside the first region, and
\[
    z^* \ = \ \frac{\varepsilon}{R_{uv}}.
\]
Theorem~\ref{thm:local-braess} now gives
\[
    C_{\mathrm{new}}-C_{\mathrm{old}}
    \ =\ 
    -V_{uv}z^*
    \ =\ 
    \frac{(-V_{uv})\varepsilon}{R_{uv}}
    \ >\ 0.
\]
Thus the chord induces Braess' Paradox.
\end{proof}

\section{Consequences of the Braess Criterion}
\label{sec:consequences}
\subsection{Harmful Chords}
Let the additional edge $u\rightarrow v$ be a \emph{shortcut} if its endpoints satisfy
\[
i'\ \geq\  i
\qquad\text{and}\qquad
j'\ \geq\  j,
\]
where $u=(i,j)$ and $v=(i',j')$. By repeatedly applying Lemma~\ref{lem:monotone}, such a chord has a positive voltage drop $V_{uv}$ and therefore does not satisfy the Braess criterion. 

\begin{observation}
\label{obs:shortcut}
A shortcut chord cannot induce Braess' Paradox, nor can any chord with level gain $k\le0$. Hence only chords moving forward in one coordinate and backward in another can produce the paradox.
\end{observation}

The edge-usage condition gives a threshold in the total demand. Define the
unit-demand voltage coefficient
\[
\gamma_{uv}
\ :=\ 
(\ee_s-\ee_t)^{\mathsf T}\Lap^+(\ee_u-\ee_v),
\]
so that
\[
V_{uv}\ =\ q\gamma_{uv}.
\]

\begin{corollary}[Demand threshold]
\label{cor:demand-threshold}
Suppose $\gamma_{uv}<0$. If $bk\le d$, then the chord is unused for
every demand $q>0$. If $bk>d$, then the chord carries positive flow
if and only if
\[
q\ <\ 
\frac{bk-d}{-\gamma_{uv}}.
\]
Consequently, a chord with negative original-grid voltage cannot
induce Braess' Paradox once the demand reaches this threshold.
\end{corollary}

\begin{proof}
By Corollary~\ref{cor:use}, the chord is used precisely when
\[
q\gamma_{uv}+bk \ >\ d.
\]
If $bk\le d$, this is impossible because $\gamma_{uv}<0$. If $bk>d$,
rearranging gives
\[
q\ <\ 
\frac{bk-d}{-\gamma_{uv}}.
\]
An unused chord does not alter the equilibrium and therefore cannot
induce the paradox.
\end{proof}

The first-region formula also gives an exact parameter range in which
the paradox is completely determined. 

\begin{corollary}[First-region parameter window]
\label{cor:first-region-window}
Suppose $V_{uv}<0$ and $k>0.$
The equilibrium flow lies in the first region
\[
0<z^*<\overline{z}
\]
if and only if
\[
d-V_{uv}
\ <\ 
bk
\ <\ 
d-V_{uv}+(R_{uv}+c)\overline{z}.
\]
Whenever these inequalities hold, the chord induces Braess' Paradox,
and
\[
C_{\mathrm{new}}-C_{\mathrm{old}}
\ =\ 
\frac{(-V_{uv})(V_{uv}+bk-d)}{R_{uv}+c}.
\]
\end{corollary}

\begin{proof}
In the first flow region,
\[
z^*
\ =\ 
\frac{V_{uv}+bk-d}{R_{uv}+c}.
\]
Since $R_{uv}+c>0$, the condition
\[
0\ <\ z^*\ <\ \overline{z}
\]
is equivalent to
\[
0\ <\ V_{uv}+bk-d<(R_{uv}+c)\overline{z},
\]
which gives the displayed inequalities. Because $V_{uv}<0$,
Theorem~\ref{thm:local-braess} shows that the chord induces
Braess' Paradox. Substituting the formula for $z^*$ into
\[
C_{\mathrm{new}}-C_{\mathrm{old}}
\ =\ 
-V_{uv}z^*
\]
gives the result.
\end{proof}

\subsection{Prevalence}

Probabilistic, asymptotic prevalence results on Braess' Paradox are known for random-graph and expander models \cite{valiantroughgarden2010, chungyoungzhao2012}; the grid structure permits an exact count. Specifically, we calculate the proportion of chords $u \rightarrow v$ with $V_{uv} \ <\  0$ and $k \ >\  0$ for all grids with $1 \le m,n \le 100$. 

Scaling $a$ and $q$ by constant factors does not change the sign of the desired quantities; we assume $a = q = 1$. 

Using the spectral decomposition of the grid Laplacian $\Lap$, we compute the single voltage vector $\phi = \Lap^+ (\ee_s - \ee_t)$. All voltage drops are then obtained from $V_{uv} = \phi_u - \phi_v$. This allows us to calculate the matrix of voltage drops more efficiently than building the Laplacian and using built-in pseudoinverse functions.

The path graph on $N$ vertices is denoted by $P_N$. Its Laplacian $\Lap_N$ \cite{spielman2019} has eigenvalues
$\lambda_r = 4\sin^{2}\left(\frac{\pi r}{2N}\right)$, $r=0,\dots,N-1$, with
orthonormal eigenvectors given by
\[
  \psi_r(i)
  \ =\  \sqrt{\frac{2-\delta_{r}}{N}}
    \cos\left(\frac{\pi r (i+\frac12)}{N}\right),
  \qquad i=0,\dots,N-1.
\] 

Here \[\delta_r = \begin{cases}
    1 \qquad \text{ if } r = 0 \\
    0 \qquad \text{ if } r \neq 0.
\end{cases}\]The grid $G_{m,n}$ is a Cartesian
product of paths $P_{m+1} \square P_{n+1}$, so its Laplacian is the Kronecker sum \cite{merris1998}
$\Lap = \Lap_{m+1}\otimes I_{n+1} + I_{m+1}\otimes \Lap_{n+1}$. If
$(\psi_{r_1},\lambda_{r_1})$ and $(\xi_{r_2},\mu_{{r_2}})$ are the eigenpairs of $\Lap_{m+1}$ and
$\Lap_{n+1}$ respectively, then the eigenpairs of $\Lap$ are $(\psi_{r_1}\otimes\xi_{r_2}, \lambda_{r_1}+\mu_{r_2})$ over all pairs $(r_1,r_2)$. The pseudoinverse inverts the nonzero eigenvalues while the $(r_1,r_2)=(0,0)$ term is set to zero; writing
\[  \Gamma_{r_1,r_2} \ :=\ 
  \begin{cases}
    \left(\lambda_{r_1}+\mu_{r_2}\right)^{-1}, & (r_1,r_2)\neq(0,0),\\
    0, & (r_1,r_2)=(0,0),
  \end{cases}
\] we have
\[ 
  \Lap^{+}
  \ =\  \sum_{r_1=0}^{m}\sum_{r_2=0}^{n}
      \Gamma_{r_1,r_2} 
      \left(\psi_{r_1}\otimes\xi_{r_2}\right)\left(\psi_{r_1}\otimes\xi_{r_2}\right)^{\mathsf{T}}.
\]
This lets us compute all possible voltage drops in $O\left(|V|^2\right)$ time. There are many chords where $V_{uv} = 0$; occasionally these drops are returned as residuals of magnitude $10^{-14}$ due to floating-point errors. Meanwhile, chords genuinely capable of Braess' Paradox carry drops many orders of magnitude larger. Therefore we treat voltage drops with absolute value at most $10^{-12}$ as zero; the classification is unchanged for any threshold in $[10^{-14}, 10^{-10}]$. 

The full algorithm is available in the accompanying repository \cite{LuBraessCode}. 

\begin{figure}[H]
  \centering
  \includegraphics[width=0.8\linewidth]{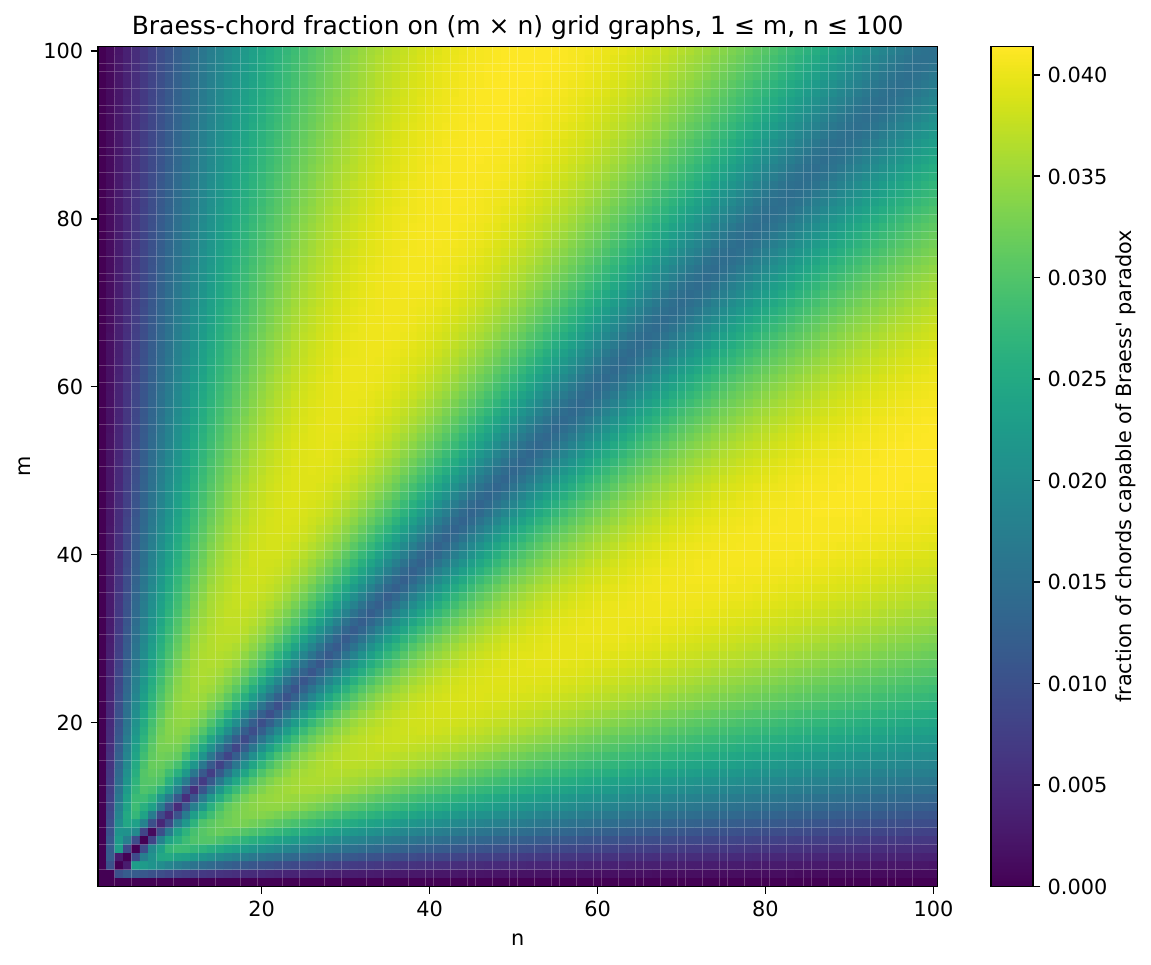}
  \caption{Fraction of Braess-capable Chords}
  \label{fig:heatmap}
\end{figure}

Figure~\ref{fig:heatmap} shows the Braess-capable chord fraction across all grids
$G_{m,n}$ with $1\le m,n\le 100$. The \hyperref[app:results]{appendix} records, for each $m$, the maximum Braess-capable fraction over all $m \leq n \leq 100$ and the value of $n$ at which it occurs. 

The fraction along $m=n$ is far below that of nearby
rectangles, so
isotropic grids are comparatively resistant to Braess-inducing chords. In fact, square grids smaller than $7 \times 7$ admit zero Braess-capable chords. The fraction is maximized along two rays of aspect ratio close to $2 : 1$. Finally, along these ridges the fraction increases with grid size but with diminishing increments,
suggesting a limiting value slightly above $0.04$; the largest
proportion in the plotted range, $\approx 0.041388$, is attained at $(m, n) = (52, 100)$ and $(100, 52)$.

\section{Severity and the 4/3 Bound}
\label{sec:bounds}

The \emph{Price of Anarchy} \cite{PoA} measures the inefficiency caused by selfish route choices and is defined by
\[
    \operatorname{PoA} \ := \ 
    \frac{C\left(f^G_{\mathrm{UE}}\right)}
         {C\left(f^G_{\mathrm{SO}}\right)}.
\]
Roughgarden and Tardos \cite{roughgardentardos2002} proved that for general graphs with nonnegative affine latency functions, the Price of Anarchy is at most $4/3$. Moreover, the original User Equilibrium is feasible in the augmented network. Hence \[C(f_{SO}^{G+e}) \ \le\  C(f_{UE}^G)\]
and 
\[
BR \ = \ \frac{C(f^{G+e}_{UE})}{C(f^G_{UE})} \ \leq\ \frac{C(f^{G+e}_{UE})}{C(f^{G+e}_{SO})} \ \le \ \frac43
\]
for any affine network. We now use the grid structure to derive a stronger bound.

\subsection{A Grid-Dependent Upper Bound}
\begin{theorem}[Braess Ratio Bound]
\label{thm:bound}
For a uniform affine $m \times n$ grid with $\max(m,n) >1$, \[BR\ \le \ 1+\frac{1}{\frac{m+n}{\max(m,n)-1} + 2\sqrt{\frac{m+n}{\max(m,n)-1}}} < \frac{4}{3}.\]
\end{theorem}

\begin{proof}

If the additional edge does not induce Braess' Paradox, the Braess Ratio is at most 1 by definition.

Suppose that the added chord induces Braess' Paradox, and let
$z^*\in(0,q]$ be its equilibrium flow. By
Theorem~\ref{thm:global-cost-bound}, $
V_{uv}<0.$

We first obtain a lower bound on the grid intercept $b$.
Equation~\eqref{eq:energy-bound} gives, for every $z\in(0,q]$,
\[
\Edir(z)-\Edir(0)
-\frac{z}{2}{\Edir'}_{-}(z)
\ \le\ 
-V_{uv}z.
\]
The directed feasible flows form a subset of the unrestricted
electrical flows, so
\[
\Edir(z)
\ \ge\ 
q^2R_{st}-2V_{uv}z+R_{uv}z^2.
\]
Since
\[
\Edir(0)\ =\ q^2R_{st},
\]
we have
\[
\Edir(z)-\Edir(0)
\ \ge\ 
-2V_{uv}z+R_{uv}z^2.
\]
Combining the preceding two inequalities gives
\[
-2V_{uv}z+R_{uv}z^2
-\frac{z}{2}{\Edir'}_{-}(z)
\ \le\ 
-V_{uv}z.
\]
Because $z>0$, this rearranges to
\begin{equation}
\label{eq:directed-slope-lower-bound}
{\Edir'}_{-}(z)
\ \ge\ 
-2V_{uv}+2R_{uv}z.
\end{equation}

The nonconstant part of the Beckmann Potential is
\[
\frac12\Edir(z)-bkz+\frac{c}{2}z^2+dz.
\]
Since $z^*>0$ is its minimizer on $[0,q]$, its left derivative at
$z^*$ is nonpositive. Therefore
\[
\frac12{\Edir'}_{-}(z^*)
-bk+cz^*+d
\ \le\ 0,
\]
and hence
\[
bk
\ \ge\ 
\frac12{\Edir'}_{-}(z^*)+cz^*+d.
\]
Applying \eqref{eq:directed-slope-lower-bound},
\begin{align*}
bk
\ &\ge\ 
-V_{uv}+R_{uv}z^*+cz^*+d\\
\ &\ge\ 
-V_{uv}+R_{uv}z^*.
\end{align*}
A paradox-inducing chord has $k>0$, so
\begin{equation}
\label{eq:grid-intercept-lower-bound}
b
\ \ge\ 
\frac{-V_{uv}+R_{uv}z^*}{k}.
\end{equation}

The equilibrium cost in the original grid is
\[
C_{\mathrm{old}}
\ =\ 
q^2R_{st}+bq(m+n).
\]
By Theorem~\ref{thm:global-cost-bound},
\[
C_{\mathrm{new}}-C_{\mathrm{old}}
\ \le\ 
-V_{uv}z^*.
\]
Consequently,
\[
BR
\ =\ 
\frac{C_{\mathrm{new}}}{C_{\mathrm{old}}}
\ \le\ 
1+
\frac{-V_{uv}z^*}
{q^2R_{st}+bq(m+n)}.
\]
Using \eqref{eq:grid-intercept-lower-bound},
\[
BR
\ \le\ 
1+
\frac{-V_{uv}z^*}
{
q^2R_{st}
+
\frac{q(m+n)}{k}
\left(-V_{uv}+R_{uv}z^*\right)
}.
\]

The right-hand side is increasing in $z^*$. Since $z^*\ \le\  q$,
\begin{equation}
\label{eq:intermediate_bound}
BR
\ \le\ 
1+
\frac{-V_{uv}}
{
qR_{st}
+
\frac{m+n}{k}
\left(qR_{uv}-V_{uv}\right)
}.
\end{equation}

Recall that \[
R_{st}
\ =\ 
(\ee_s-\ee_t)^\mathsf{T}\Lap^+(\ee_s-\ee_t),
\]
\[
R_{uv}
\ =\ 
(\ee_u-\ee_v)^\mathsf{T}\Lap^+(\ee_u-\ee_v),
\] 
\[
V_{uv}\ =\ 
q(\ee_s-\ee_t)^\mathsf{T}\Lap^+(\ee_u-\ee_v).
\]

Since $\Lap^+$ is positive semidefinite, we may write $\Lap^+ = M^TM$ for some $M \in \mathbb{R}^{|V| \times |V|}$. Then the three quantities can be expressed as 

\[R_{st}
\ =\ 
\|M(\ee_s - \ee_t)\|^2
\]
\[
R_{uv}
\ =\ 
\|M(\ee_u -  \ee_v)\|^2,
\] 
\[
V_{uv}\ =\ 
q\left(M(\ee_s - \ee_t)\right)\cdot \left(M(\ee_u-\ee_v)\right).
\]

Cauchy--Schwarz \cite{horn-johnson} states that for two vectors $v_1$ and $v_2$ \[|v_1 \cdot v_2| \ \le\  \|v_1\| \|v_2\|.\] Taking $v_1 = M(\ee_s - \ee_t)$ and $v_2 = M(\ee_u - \ee_v),$ we have $V_{uv}^2 \le q^2R_{st}R_{uv}$.

Since $V_{uv} < 0$, \[-V_{uv} \ \le\  q\sqrt{R_{st}R_{uv}}.\] Since the right-hand side of \eqref{eq:intermediate_bound} is increasing in $-V_{uv}$, \[BR \ \le\  1+ \frac{\sqrt{R_{uv} R_{st}}}{R_{st} + \frac{m+n}{k}\left(R_{uv}+\sqrt{R_{uv}R_{st}}\right)}.\]
By the AM-GM inequality \[R_{st} + \frac{m+n}{k}R_{uv} \ \ge\  2\sqrt{\frac{m+n}{k}}\sqrt{R_{uv}R_{st}}.\] Therefore \begin{align*}BR \ &\le\  1 + \frac{\sqrt{R_{uv}R_{st}}}{\frac{m+n}{k}\sqrt{R_{uv}R_{st}} + \sqrt{\frac{m+n}{k}} \sqrt{R_{uv}R_{st}}} \\
& \ =\  1 + \frac{1}{\frac{m+n}{k} + 2\sqrt{\frac{m+n}{k}}}. \end{align*}
 
Observation~\ref{obs:shortcut} shows that a Braess-capable chord increases in one coordinate while decreasing the other. The increase is at most $\max(m,n)$, while the decrease is at least 1. Therefore its level gain $k$ satisfies \[k \ \le\  \max(m,n) - 1.\]
 
Hence \begin{equation}
    \label{eq:best_bound}
    BR \ \le \  1+\frac{1}{\frac{m+n}{\max(m,n)-1} + 2\sqrt{\frac{m+n}{\max(m,n)-1}}}.
\end{equation}
In particular, since $\max(m,n) - 1 < m+n$, we have \[BR \ < \ 1 + \frac{1}{3} = \frac{4}{3}.\] Thus the grid structure yields a strict improvement over the general $\frac{4}{3}$ bound.
\end{proof}

\begin{corollary}
    Without loss of generality, assume $m \ge n$, and write the aspect ratio $r := n/m$. Then for each $r$, as $m \rightarrow \infty$, \[BR \ \le \ \frac{4}{3} - \frac{2}{9}r  + \frac{19}{108}r^2+O(r^3)  \qquad \text{as }r \rightarrow 0.\]
\end{corollary}
\begin{proof}
    Write the bound from \eqref{eq:best_bound} as \[BR \ \le \ 1 + \frac{1}{\frac{1+r}{1 - 1/m} + 2\sqrt{\frac{1+r}{1-1/m}}}. \]
    The $1/m$ term is negligible; hence, it suffices to expand \[1 + \frac{1}{1+r+2\sqrt{1+r}}.\] Using the Binomial Theorem, we may write $\sqrt{1+r} \ = \ 1 + \frac12r-\frac18r^2 + \frac{1}{16}r^3 + O(r^4)$.

    In other words, \begin{align*}
        BR \ &\le\  1 + \frac{1}{3+2r - \frac14r^2 + \frac18r^3 + O(r^4)} \\
        &=\  \frac{4}{3} - \frac{2}{9}r + \frac{19}{108}r^2 + O(r^3) \qquad \text{as }r \rightarrow 0,
    \end{align*} as claimed. 
\end{proof}

\begin{figure}[H]
\centering
\includegraphics{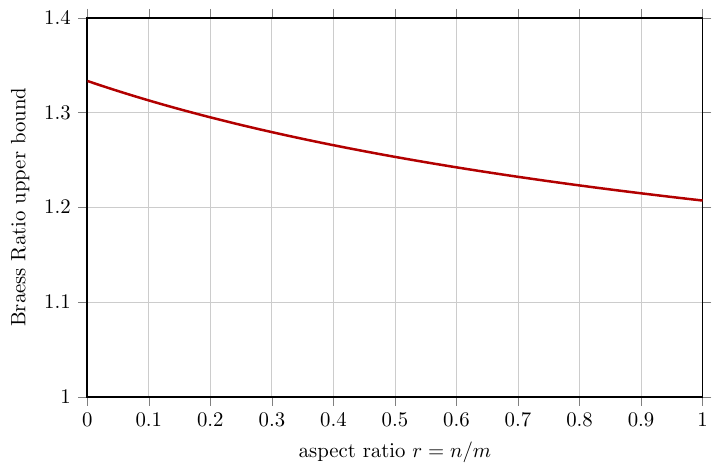} 
\caption{The upper bound in the limit $m\to\infty$, as a function of aspect ratio.}
\label{fig:brbound}
\end{figure}

As shown in Figure~\ref{fig:brbound}, for thin rectangles, the bound in \eqref{eq:best_bound} approaches $4/3$; for square grids, it approaches $1 + 1/(2+2\sqrt2) \approx 1.207$. Thus, the bound is substantially stronger than the general affine bound for square grids, while asymptotically recovering the general bound on thin rectangles. For comparison, for arbitrary continuous and nondecreasing latency functions, the Braess Ratio from adding one edge is at most 2 \cite{lin2011}.

\subsection{The Optimal Latency}

We now establish the latency coefficients of a ratio-maximizing chord.

\begin{theorem}[Optimal added edge latency]
\label{thm:optimal-edge}
On an $m \times n$ grid capable of Braess' Paradox, a chord maximizes the Braess Ratio only when it has zero latency; that is,
\[
  c \ =\  d\  =\  0.
\]
\end{theorem}

\begin{proof}
Fix an added chord $u \rightarrow v$ and suppose that $b_1, c_1, d_1 \ge 0$ with $(c_1, d_1) \neq (0,0)$ induce Braess' Paradox with equilibrium chord flow $z^{*} > 0$. By Theorem~\ref{thm:capability}, $V_{uv} < 0$ and $k > 0$. Recall that $\Edir(0) = q^{2}R_{st}$,
and since $V_{uv} < 0$, inequality~\eqref{eq:relaxation-bound} gives
\[
  \Edir(z^{*}) - \Edir(0)
  \ \ge\ 
  -2V_{uv}z^{*} + R_{uv}(z^{*})^{2}
  \ >\  0.
\]

The nonconstant part of the Beckmann Potential is
\[
  \frac{1}{2}{\Edir}(z) - b_1kz + \frac{c_1}{2}z^{2} + d_1z.
\]
Since $z^{*}$ minimizes this convex function on $[0, q]$, its left
derivative there is nonpositive, giving
\begin{equation}
  \frac{1}{2}{\Edir'}_-(z^{*}) - b_1k + c_1z^{*} + d_1 \ \le\  0.
  \label{eq:leftderiv}
\end{equation}
Moreover,
\[
  {\Edir'}_+(0) = -2V_{uv} > 0,
\]
so convexity gives ${\Edir'}_-(z^{*}) > 0$.

Now set
\[
  b_{0} \ =\  \frac{{\Edir'}_-(z^{*})}{2k},
  \qquad
  c_{0} \ =\  d_{0} \ =\  0.
\]

Under the parameters $(b_0,0,0)$, the left derivative of the
Beckmann Potential at $z^*$ is
\[
\frac{1}{2}{\Edir'}_{-}(z^*)-b_0k\ =\ 0.
\]
By convexity, $z^*$ is also the minimizer of the Beckmann Potential under these new parameters.

We now compare the two travel-time increases under $(b_1, c_1, d_1)$ and $(b_0, 0,0)$. Under the original parameters $(b_1,c_1,d_1)$,
\[
  \Delta_{\mathrm{1}}
  \ := \
  C_{\mathrm{new}}(b_1, c_1, d_1) - C_{\mathrm{old}}(b_1)
  \ = \
  \Edir(z^{*}) - \Edir(0) - z^{*}\left(b_1k - c_1z^{*} - d_1\right),
\]
while under $(b_{0}, 0, 0)$,
\[
  \Delta_{0}
  \ := \
  C_{\mathrm{new}}(b_0, 0, 0) - C_{\mathrm{old}}(b_0)
  \ = \
  \Edir(z^{*}) - \Edir(0) - \frac{z^{*}}{2}{\Edir'}_{-}(z^{*}).
\]
By~\eqref{eq:leftderiv},
\begin{equation}
  \Delta_{0} \ \ge\  \Delta_{1} \ >\  0,
  \label{eq:excess}
\end{equation}
the strict positivity holding because the original chord induces the
paradox.

Next we compare denominators. From~\eqref{eq:leftderiv}, and since
$c_1, d_1 \ge 0$, $z^{*} > 0$, and $k > 0$,
\[
  b_{0}
  \ =\  \frac{{\Edir'}_-(z^{*})}{2k}
  \ \le\  \frac{b_1k - c_1z^{*} - d_1}{k}
  \ <\  b_1.
\]
Writing $C_{\mathrm{old}}(b) = q^{2}R_{st} + bq(m+n)$ for the
original-grid cost, it follows $C_{\mathrm{old}}(b_0) < C_{\mathrm{old}}(b_1)$.

Finally, we express the Braess Ratio and apply
\eqref{eq:excess}:
\[
  BR_{0}
  \ =\  1 + \frac{\Delta_{0}}{C_{\mathrm{old}}(b_{0})}
  \ \ge\ 
  1 + \frac{\Delta_{1}}{C_{\mathrm{old}}(b_{0})}
  \ >\ 
  1 + \frac{\Delta_{1}}{C_{\mathrm{old}}(b_1)}
  \ =\  BR_{\mathrm{1}}.
\]

Thus every paradox-inducing parameter choice $b_1, c_1, d_1 \ge 0$ with $(c_1, d_1) \neq (0,0)$ can be replaced by one with
$c = d = 0$ while strictly increasing the Braess Ratio. Hence any maximizing choice must give the added chord zero latency.
\end{proof}

\section{Future Directions}
\label{sec:future}

Although Theorem~\ref{thm:bound} provides an upper bound, we have not determined the supremum for the Braess Ratio. Theorem~\ref{thm:optimal-edge} shows that any maximizing chord has zero latency, but further work is needed to determine whether a maximizing configuration exists and, if so, the location of the chord and dimensions of the maximizing grid. Establishing a tighter upper bound decreasing in $m+n$ that could reduce the global supremum to a finite check is a natural next step.

Although no elementary closed form has been found for the voltage drops in a grid, the two-point effective resistance of the uniform grid does
admit an exact finite summation. Using the closed form for the grid Laplacian, Wu~\cite{wu2004} showed that for a grid with uniform resistance $R$,
\begin{align*}
    &R_{\mathrm{eff}}\left((i, j), (i', j')\right) \\
   &= R \sum_{\substack{0 \le r_1 \le m,\ 0 \le r_2 \le n \\ (r_1, r_2) \ne (0, 0)}}
    \frac{\varepsilon_{r_1}  \varepsilon_{r_2}}{(m + 1)(n + 1)} 
    \frac{\left[
      \cos\frac{\pi r_1 (i + \frac{1}{2})}{m + 1}
      \cos\frac{\pi r_2 (j + \frac{1}{2})}{n + 1}
      - \cos\frac{\pi r_1 (i' + \frac{1}{2})}{m + 1}
      \cos\frac{\pi r_2 (j' + \frac{1}{2})}{n + 1}
    \right]^2}
    {4 - 2 \cos\frac{\pi r_1}{m + 1} - 2 \cos\frac{\pi r_2}{n + 1}},
\end{align*}
where the Neumann normalization factor is
\[
\varepsilon_r \ :=\ 
\begin{cases}
1, & r = 0, \\
2, & r \ge 1.
\end{cases}
\]
In particular, for opposite corners $(0, 0)$ and $(m, n)$, the expression simplifies to
\[
R^{\mathrm{corner}}_{m+1,  n+1}
  \ =\  R \sum_{\substack{0 \le r_1 \le m,\ 0 \le r_2 \le n \\ r_1 + r_2 \text{ odd}}}
    \frac{\varepsilon_{r_1}  \varepsilon_{r_2}}{(m + 1)(n + 1)} 
    \frac{\cos^2\frac{\pi r_1}{2(m + 1)} \cos^2\frac{\pi r_2}{2(n + 1)}}
         {\sin^2\frac{\pi r_1}{2(m + 1)} + \sin^2\frac{\pi r_2}{2(n + 1)}}.
\]
Essam and Wu~\cite{essamwu2009} derived the asymptotic expansion for
an $n \times n$ vertex grid, or $G_{n-1, n-1}$:
\[
R^{\mathrm{corner}}_{n, n}
  \ =\  R \left( \frac{4}{\pi} \log n + C_0
    + \frac{C_2}{n^2} - \frac{C_4}{n^4} + O(n^{-6}) \right),
\] where \[C_0 \approx 0.077318, \qquad C_2 \approx 0.266070, \quad \text{ and } \quad  C_4 \approx 0.534779.\]
These exact summations may be applied to derive asymptotics for $R_{st}$, $R_{uv}$, and $V_{uv}$ as $m, n \rightarrow \infty$. The proof of Theorem~\ref{thm:bound} in fact gives a stricter chord-dependent bound
\[
    BR \ \leq\  1+
    \frac{1}{
        \frac{m+n}{k}
        +2\sqrt{\frac{m+n}{k}}
    }.
\]
A stronger relation between the level gain
$k$ and the voltage condition $V_{uv}<0$ that bounds $k$ further could produce a tighter bound. Furthermore, an asymptotic expansion for $V_{uv}$ could help classify all grids that admit Braess' Paradox.  It would suffice to determine the grids for which there exists a chord satisfying the capability conditions of Theorem~\ref{thm:capability}.

Another possible extension is to construct a physical circuit that shows Braess' Paradox in the grid setting. Physical realizations of the paradox are known: Cohen and Horowitz \cite{cohenhorowitz1991} demonstrated electrical and mechanical analogues for a four-node network in 1991. Although Lemma~\ref{thm:linear-latency} shows that the paradox cannot occur in a network of only resistors, circuit elements such as batteries can be added to produce the constant term of the affine latencies. Constructing such a circuit could further illustrate the connection between congestion networks and electrical systems. 
\section*{Appendix}
We provide computational results for every $m \times n$ grid with $1 \le m \le n \le 100$, recording the proportion of chords capable of Braess' Paradox for some combination of affine latency functions. For each fixed $m$, Table~\ref{tab:prevalence-results} records the maximum proportion as well as the value of $n$ achieving the maximum.
\label{app:results}

\begingroup
\small
\setlength{\tabcolsep}{10pt}
\renewcommand{\arraystretch}{1.05}

\begin{longtable}{c|c|c}
\caption{Maximum Braess-capable chord fraction.}
\label{tab:prevalence-results}\\
\hline
$m$ & Maximum fraction & Maximizing $n$ \\
\hline
\endfirsthead

\multicolumn{3}{c}%
{\tablename\ \thetable\ -- continued from previous page}\\
\hline
$m$ & Maximum fraction & Maximizing $n$ \\
\hline
\endhead

\hline
\multicolumn{3}{r}{Continued on next page}
\endfoot

\hline
\endlastfoot

1 & 0.0000000 & all $n$ \\
2 & 0.0120482 & 4 \\
3 & 0.0231092 & 5 \\
4 & 0.0250000 & 9 \\
5 & 0.0283251 & 10 \\
6 & 0.0302542 & 14 \\
7 & 0.0317882 & 15 \\
8 & 0.0329175 & 17 \\
9 & 0.0337430 & 19 \\
10 & 0.0345157 & 22 \\
11 & 0.0351084 & 23 \\
12 & 0.0358868 & 25 \\
13 & 0.0363279 & 26 \\
14 & 0.0366876 & 29 \\
15 & 0.0371111 & 31 \\
16 & 0.0375586 & 33 \\
17 & 0.0377757 & 36 \\
18 & 0.0380794 & 38 \\
19 & 0.0383396 & 40 \\
20 & 0.0384961 & 40 \\
21 & 0.0387346 & 42 \\
22 & 0.0389360 & 44 \\
23 & 0.0391416 & 46 \\
24 & 0.0392817 & 49 \\
25 & 0.0394058 & 50 \\
26 & 0.0395753 & 52 \\
27 & 0.0396785 & 55 \\
28 & 0.0398001 & 56 \\
29 & 0.0399273 & 58 \\
30 & 0.0400402 & 59 \\
31 & 0.0401404 & 62 \\
32 & 0.0402516 & 64 \\
33 & 0.0403417 & 66 \\
34 & 0.0404095 & 68 \\
35 & 0.0404916 & 70 \\
36 & 0.0405609 & 71 \\
37 & 0.0406467 & 73 \\
38 & 0.0407100 & 74 \\
39 & 0.0407693 & 78 \\
40 & 0.0408399 & 79 \\
41 & 0.0408947 & 81 \\
42 & 0.0409509 & 83 \\
43 & 0.0410109 & 85 \\
44 & 0.0410633 & 87 \\
45 & 0.0411109 & 88 \\
46 & 0.0411490 & 91 \\
47 & 0.0412011 & 92 \\
48 & 0.0412480 & 94 \\
49 & 0.0412920 & 97 \\
50 & 0.0413250 & 99 \\
51 & 0.0413705 & 100 \\
52 & 0.0413880 & 100 \\
53 & 0.0413795 & 100 \\
54 & 0.0413402 & 100 \\
55 & 0.0412709 & 100 \\
56 & 0.0411695 & 100 \\
57 & 0.0410459 & 100 \\
58 & 0.0408774 & 100 \\
59 & 0.0406964 & 100 \\
60 & 0.0404781 & 100 \\
61 & 0.0402309 & 100 \\
62 & 0.0399576 & 100 \\
63 & 0.0396483 & 100 \\
64 & 0.0393181 & 100 \\
65 & 0.0389557 & 100 \\
66 & 0.0385692 & 100 \\
67 & 0.0381601 & 100 \\
68 & 0.0377155 & 100 \\
69 & 0.0372487 & 100 \\
70 & 0.0367473 & 100 \\
71 & 0.0362263 & 100 \\
72 & 0.0356794 & 100 \\
73 & 0.0351114 & 100 \\
74 & 0.0345059 & 100 \\
75 & 0.0338793 & 100 \\
76 & 0.0332348 & 100 \\
77 & 0.0325730 & 100 \\
78 & 0.0318800 & 100 \\
79 & 0.0311695 & 100 \\
80 & 0.0304345 & 100 \\
81 & 0.0296854 & 100 \\
82 & 0.0289135 & 100 \\
83 & 0.0281214 & 100 \\
84 & 0.0273185 & 100 \\
85 & 0.0264946 & 100 \\
86 & 0.0256635 & 100 \\
87 & 0.0248210 & 100 \\
88 & 0.0239655 & 100 \\
89 & 0.0231069 & 100 \\
90 & 0.0222428 & 100 \\
91 & 0.0213778 & 100 \\
92 & 0.0205159 & 100 \\
93 & 0.0196696 & 100 \\
94 & 0.0188284 & 100 \\
95 & 0.0180276 & 100 \\
96 & 0.0172485 & 100 \\
97 & 0.0165278 & 100 \\
98 & 0.0158841 & 100 \\
99 & 0.0153608 & 100 \\
100 & 0.0150740 & 100 \\

\end{longtable}
\endgroup

\end{document}